\documentclass[letterpaper,10pt]{amsart}
\usepackage[utf8]{inputenc}

\usepackage{amssymb}
\usepackage[active]{srcltx}
\usepackage[all]{xy}
\usepackage{enumerate}
\usepackage{float}
\usepackage{paralist}
\usepackage{cite}
\usepackage{xspace}
\usepackage{multirow}
\usepackage{array}
\usepackage{tikz}
\usepackage{pgf}
\usetikzlibrary{arrows}
\usetikzlibrary{positioning}
\usepackage{pgfkeys}

\usepackage[english]{babel}

\usepackage[colorlinks]{hyperref}

\usepackage{xcolor}
\newcommand{\numberset}[1]{\ensuremath{\mathbb{#1}}} 
\newcommand{\calm}[1]{\mathcal{#1}}
\newcommand{\N}{\numberset{N}} 
\newcommand{\R}{\numberset{R}} 
\newcommand{\C}{\numberset{C}} 
\newcommand{\Sp}[1]{\ensuremath{\mathbb{S}^{#1}}}

\newcommand{\D}{\mathbb{D}}

\newcommand{\BD}[1]{\mathbf{D}^{#1}}
\newcommand{\pr}[1]{#1^{-1}}
\newtheorem{thm}{Theorem}[section]
\newtheorem{prop}[thm]{Proposition}

\newtheorem{cor}[thm]{Corollary}
\theoremstyle{definition}

\newtheorem{exa}[thm]{Example}

\newcommand{\ie}{\textit{i.e.},\xspace}

\begin{document}
\title[Cyclic branched covers as links of singularities]{Cyclic branched covers of knots as links of real isolated singularities}
\author[H. Aguilar-Cabrera]{Hayd\'ee Aguilar-Cabrera}
\address{Instituto de Ci\^encias Matem\'aticas e de Computa\c{c}\~ao\\
Universidade de S\~{a}o Paulo\\ 
Avenida Trabalhador são-carlense, 400\\
Centro, CEP 13566-590\\
S\~{a}o Carlos, S\~{a}o Paulo, Brazil\\
}
\email{aguilarcabrerah@gmail.com}
\date{November 5, 2013}
\thanks{Research partially supported by CONACyT grant ``Estancias Posdoctorales y Sab\'aticas al Extranjero para la Consolidaci\'on de Grupos de Investigaci\'on''. It is greatly appreciated the hospitality of the Department of Mathematics of Columbia University and the Instituto de Matem\'aticas, Unidad Cuernavaca of the Universidad Nacional Aut\'onoma de M\'exico.}
\keywords{Real singularities, Milnor fibration, Figure-eight knot, Hyperbolic manifold, Hantzsche-Wendt manifold}
\subjclass[2010]{Primary 32S55; Secondary 57M10,57M25,57Q45,26C99}

\begin{abstract}
Given a real analytic function $f$ from $\R^4$ to $\R^2$ with isolated critical point at the origin, the link $L_f$ of the singularity is a real fibred knot in $\Sp{3}$. From this singularities, we construct a family of real isolated suspension singularities from $\mathbb{R}^6$ to $\mathbb{R}^2$ such that its links are the total spaces of the $n$-branched cyclic coverings over the corresponding knots. In this way we obtain as links of singularities, $3$-manifolds that does not appear in the complex case, such as hyperbolic $3$-manifolds or the Hantzsche-Wendt manifold.
\end{abstract}

\maketitle

\section{Introduction}

In the study of the topology of singularities there is an interaction of several branches of mathematics such
as knot theory, low dimensional topology, etc.

In the case of complex singularities, it is known that the $3$-manifolds that appear in the study of the topology of these singularities are of a special type: graph manifolds. Graph manifolds were defined and classified by Waldhausen in his thesis \cite{MR0236930} and his work together with Grauert's, Mumford's and others, gives an accurate description of what manifolds are links of complex singularities. There is also an algorithmic form of this description made by Neumann in \cite{Neu:calcplumb}.

Moreover, the $3$-manifolds that appear as links of complex singularities are always prime and a natural decomposition of prime $3$-manifolds is the JSJ-decomposition. From this point of view, a graph manifold is simply a $3$-manifold that only have Seifert-fibred JSJ-pieces. Let us recall that the relevant geometry for simple non-Seifert-fibred pieces is the hyperbolic geometry, \ie $3$-manifolds that appear as links of complex singularities are manifolds that have not hyperbolic pieces in their geometric decomposition.

In this article we show how to construct links of real isolated singularities, that are $3$-manifolds, as cyclic covers of real algebraic fibred knots. In particular we present a family of hyperbolic manifolds as links of real singularities. We do this generalising the tools used in \cite{2012arXiv1211.5103A}.

\section{Real fibred knots}

Let $f \colon (\R^4,0) \to (\R^2,0)$ be a real analytic function with isolated critical point at the origin and let $L_f = \pr{f}(0) \cap \Sp{3}_{\varepsilon}$ be the link of the singularity, where $\Sp{3}_{\varepsilon}$ is a sphere centred at the origin, of radius $\varepsilon > 0$, with $\varepsilon$ small enough.

By the Milnor fibration theorem in its real version (\!\!\cite[Theorem~11.2]{milnor:singular}), there exists a $C^{\infty}$-locally trivial fibration
\begin{equation*}
\varphi_f \colon \Sp{3}_{\varepsilon} \setminus L_f \to \Sp{1} \ ,
\end{equation*}
\ie the link $L_f$ is a real fibred knot. Let us see an example.

\begin{exa}[Figure-eight knot]
In \cite{MR643562}, Perron proved the following result:

\begin{prop}
Let $f \colon (\R^4,0) \to (\R^2,0)$ be the polynomial map
\begin{equation*}
f(x_1,x_2,x_3,x_4)= \left(x_3 \rho^2+x_1(8x_1^2-2 \rho^2),\sqrt{2}x_4x_1+x_2(8x_1^2-\rho^2)\right) \ ,
\end{equation*}
where $\rho=x_1^2+x_2^2+x_3^2+x_4^2$ and let $g \colon (\R^4,0) \to (\R^2,0)$ be the polynomial map
\begin{equation*}
g(x_1,x_2,x_3,x_4)= f(x_1,x_2,x_3^2-x_4^2,2x_3x_4) \ .
\end{equation*}
Then $f$ and $g$ have isolated critical point at the origin and the figure-eight knot is the link of the singularity of $\pr{g}(0)$ at the origin.
\end{prop}

In his review of \cite{MR643562}, Rudolph gives a different polynomial map with isolated critical point such that the figure-eight knot is its link (see \cite{Rud82}):

\begin{prop}\label{polRudolph}
Let $f \colon (\C^2,0) \cong (\R^4,0) \to (\C,0) \cong (\R^2,0)$ be the real analytic polynomial map given by
\begin{equation*}
f(z_1,z_2)= z_2^3 -3(x_1^2+x_2^2)(1+i x_2)z_2 -2x_1 \ ,
\end{equation*}
where $z_1=x_1+i x_2$; and let $g \colon (\C^2,0) \cong (\R^4,0) \to (\C,0) \cong (\R^2,0)$ be the real analytic polynomial
\begin{equation*}
g(z_1,z_2)= f(z_1^2,z_2) \ .
\end{equation*}
Then, $g$ has isolated critical point at the origin and the figure-eight knot is the link $L_g$ of the singularity of $\pr{g}(0)$ at the origin.
\end{prop}
\end{exa}

\section{Suspension singularities}

Let $f \colon (\R^4,0) \to (\R^2,0)$ be a real analytic function with isolated critical point and let $L_f=\pr{f}(0) \cap \Sp{3}$ be the link. Now, let $F \colon (\R^{4} \times \C,0) \cong (\R^{6},0) \to (\R^2,0)$ be the real analytic function defined by
\begin{equation*}
F(x_1,x_2,x_3,x_4,z)= f(x_1,x_2,x_3,x_4) + z^r \ ,
\end{equation*}
where $(x_1,x_2,x_3,x_4) \in \R^4$, $z \in \C$ and $r \in \N$, with $r \geq 2$.

Using the following result, we obtain that the real analytic map $F$ has isolated critical point:

\begin{prop}[\!{\!\cite[Proposition 1]{MR2922705}}]\label{singaisl}
Let $h \colon (\R^{n},0) \to (\R^2,0)$ be a real analytic germ and let $r \in \N$. The analytic germ  $H \colon (\R^{n} \times \C,0) \cong (\R^{n+2},0) \to (\R^2,0)$ defined for $(x_1, \ldots, x_n) \in \R^n$ and $z \in \C$ by
\begin{equation*}
H(x_1, \ldots, x_n, z)=h(x_1, \ldots, x_n) +z^r
\end{equation*}
has an isolated singularity at the origin if and only if $h$ has an isolated singularity at the origin.
\end{prop}

\begin{proof}
The Jacobian matrix $M$ of $h(x_1, \ldots, x_n)+z^r$ with respect to the coordinates $x_1, \ldots, x_n, z, \bar{z}$ is given by
\begin{equation*}
\begin{pmatrix}
\begin{matrix}
Dh(x_1, \ldots, x_n)
\end{matrix}
\left|
\begin{matrix}
\frac{1}{2}rz^{r-1}  & \frac{1}{2}r\bar{z}^{r-1} \\[7pt]
\frac{1}{2i}rz^{r-1} & -\frac{1}{2i} r\bar{z}^{r-1}\\
\end{matrix}
\right.
\end{pmatrix}
\end{equation*}
where $Dh(x_1, \ldots, x_n)$ is the Jacobian matrix of $h$ with respect to the coordinates $x_1, \ldots, x_n$.

Let $\mathcal{P} \colon \R^n \times \C \to \R^n$ be the projection defined by $\mathcal{P}(x_1, \ldots, x_n, z)=(x_1, \ldots, x_n)$.

If $h$ has an isolated singularity at the origin, then $Dh(x_1, \ldots, x_n)$ has rank $2$ in a neighbourhood $W$ of the origin except the origin. Let $(x_1, \ldots, x_n, z) \in \mathcal{P}^{-1}(W)$. If $(x_1, \ldots, x_n) \neq 0$, then $Dh(x_1, \ldots, x_n)$ has rank two. Otherwise $z \neq 0$ and the matrix
\begin{equation*}
\begin{pmatrix}
\frac{1}{2}rz^{r-1}  & \frac{1}{2}r\bar{z}^{r-1} \\[7pt]
\frac{1}{2i}rz^{r-1} & \frac{-1}{2i} r\bar{z}_{n+1}^{r-1}\\
\end{pmatrix}
\end{equation*}
has rank two. Then $h(x_1, \ldots, x_n)+z^r$ has rank two at each point of $\mathcal{P}^{-1}(W)\setminus \{0\}$.

If $h+z^r$ has an isolated singularity at the origin, then $M$ has rank $2$ in a neighbourhood $U$ of the origin except at the origin itself; in particular $M$ has rank $2$ at the points $(x_1, \ldots, x_n,0) \in U \setminus \{0\}$; then the matrix $Dh(x_1, \ldots, x_n, z)$ has rank $2$ in the neighbourhood $\mathcal{P}(U)$ of the origin except at the origin.
\end{proof}

Let $L_F= \Sp{5} \cap \pr{F}(0)$ denote the link of the singularity of $\pr{F}(0)$ at the origin. Let $\varepsilon > 0$ be such that any sphere $\Sp{5}_{\eta}$ with $0 < \eta \leq \varepsilon$ intersects transversely $\pr{F}(0)$. Let $\varepsilon^{\prime} > 0$ be such that for all $(x_1,x_2,x_3,x_4,z) \in \pr{F}(0)$ with $(x_1,x_2,x_3,x_4) \in \D^4_{\varepsilon^{\prime}}$ we have 
\begin{equation*}
|f(x_1,x_2,x_3,x_4)|^{1/r} < \varepsilon \ .
\end{equation*}
Let us consider the polydisc
\begin{equation*}
\BD{6} = \{(x_1,x_2,x_3,x_4,z)\, |\, (x_1,x_2,x_3,x_4) \in \D^4_{\varepsilon^{\prime}}, |z| \leq \varepsilon \} \ .
\end{equation*}

By \cite[Proposition~1.7 and Application~3.8]{durfee:neighalgs} the link $L_F$ is homeomorphic to the intersection $\pr{F}(0) \cap \partial \BD{6}$. In the sequel we will also denote this intersection by $L_F$.

Let $\calm{P} \colon L_F \to \Sp{3}_{\varepsilon'}$ be the projection defined by
\begin{equation*}
\calm{P}(x_1,x_2,x_3,x_4,z)=(x_1,x_2,x_3,x_4)
\end{equation*}
and let $L' = \pr{\calm{P}}(L_f)$. Given a point $(x_1,x_2,x_3,x_4,z) \in L_F \setminus L'$, we have that $f(x_1,x_2,x_3,x_4)+ z^r = 0$, \ie $-f(x_1,x_2,x_3,x_4)=z^r$. Let $\theta= \arg \left(-f(x_1,x_2,x_3,x_4)\right)$, then $(x_1,x_2,x_3,x_4,z)$ is of the form
\begin{equation*}
\left(x_1,x_2,x_3,x_4, |\!-\!\!f(x_1,x_2,x_3,x_4)|^{1/r} e^{i \frac{\theta + 2 \pi k}{r}} \right)
\end{equation*}
for some $k$ such that $0 \leq k \leq r-1$.

We use the Milnor fibration $\varphi_f \colon \Sp{3} \setminus L_f \to \Sp{1}$ to define a map $\varphi' \colon L_F \setminus L' \to \Sp{1}$. Let $\sigma = \arg \left(-\varphi_f(x_1,x_2,x_3,x_4) \right)$, then $\varphi'$ is defined by:
\begin{equation*}
\varphi'\left(x_1,x_2,x_3,x_4, |\!-\!\!f(x_1,x_2,x_3,x_4)|^{1/r} e^{i \frac{\theta + 2 \pi k}{r}}\right) = e^{i \frac{\sigma + 2 \pi k}{r}} \ ,
\end{equation*}
with $0 \leq k \leq r-1$ fixed.

Let us define $\rho_r \colon \C \to \C$ by $\rho_r(z)=z^r$, then we have the following result.

\begin{thm}\label{brcov}
Given the maps $\calm{P}$, $\varphi_f$, $\rho_r$ and $\varphi'$, we have that
\begin{enumerate}[1)]
\item  the following diagram commutes:\label{1}
\begin{equation}\label{dgcm}
\xymatrix{
L_F \setminus L' \ar[d]_{\varphi'} \ar[r]^{\calm{P}} & \Sp{3} \setminus L_f \ar[d]^{\varphi_f} \\
\Sp{1} \ar[r]^{-\rho_r} & \Sp{1} \\
}
\end{equation}
\item $\calm{P} \colon L_F \to \Sp{3}$ is a $r$-fold cyclic branched covering with ramification locus $L_f$ and the restriction\label{2}
\begin{equation*}
\calm{P} \colon L' \to L_f
\end{equation*}
is a homeomorphism.
 \end{enumerate}
\end{thm}
\begin{proof}
Let us prove \ref{1}). Let $(x_1,x_2,x_3,x_4,z) \in L_F \setminus L'$. Then
\begin{equation*}
0 \neq z^r=-f(x_1,x_2,x_3,x_4)
\end{equation*}
and $(x_1,x_2,x_3,x_4,z) = \left(x_1,x_2,x_3,x_4, |\!-\!\!f(x_1,x_2,x_3,x_4)|^{1/r} e^{i \frac{\theta + 2 \pi k'}{r}} \right)$ with $\theta= \\ \arg(-f(x_1,x_2,x_3,x_4))$ and some $k'$ such that $0 \leq k' \leq r-1$. Then,
\begin{align*}
\varphi_f \left(\calm{P}(x_1,x_2,x_3,x_4,z) \right) &=  \varphi_f \left( \calm{P} \left( x_1,x_2,x_3,x_4, |\!-\!\!f(x_1,x_2,x_3,x_4)|^{1/r} e^{i \frac{\theta + 2 \pi k'}{r}} \right) \right) \\
 & = \varphi_f (x_1,x_2,x_3,x_4) \ .
\end{align*}
On the other hand, 
\begin{align*}
-\rho_r & \left( \varphi' \left( x_1,x_2,x_3,x_4, |\!-\!\!f(x_1,x_2,x_3,x_4)|^{1/r} e^{i \frac{\theta + 2 \pi k'}{r}} \right) \right) \\
& \quad =  -\rho_r \left( e^{i \frac{\sigma + 2 \pi k'}{r}} \right) \\
& \quad = \varphi_f(x_1,x_2,x_3,x_4) \ ,
\end{align*}
where $\sigma=\arg(-\varphi_f(x_1,x_2,x_3,x_4))$.

Hence $\varphi_f\left(\calm{P}(x_1,x_2,x_3,x_4,z)\right)=-\rho_r \left(\varphi'(x_1,x_2,x_3,x_4,z)\right)$.
 
Now, in order to prove \ref{2}), first we prove that diagram \eqref{dgcm} is a pull-back diagram.
 
Let $Q$ be the pull-back of $\varphi_f$ by $-\rho_r$ defined by
\begin{align}
Q &= \{(x_1,x_2,x_3,x_4, \lambda) \in (\Sp{3} \setminus L_f) \times  \Sp{1} \mid \varphi_f(x_1,x_2,x_3,x_4) = -\rho_r(\lambda)\} \notag \\
&= \{(x_1,x_2,x_3,x_4, \lambda) \in (\Sp{3} \setminus L_f) \times  \Sp{1} \mid \varphi_f(x_1,x_2,x_3,x_4) = -\lambda^r \} \label{eq:pull-back}
\end{align}

Then, by the universal property of the pull-back, there exist a unique map $p \colon (L_F \setminus L') \to Q$ such that the following diagram commutes:
\begin{equation*}
\xymatrix{
 L_F \setminus L' \ar@/_/[dddr]_{\varphi'} \ar[dr]^-{p} \ar@/^/[drrr]^{\calm{P}} & & &\\
                & Q \ar[dd]^{\nu} \ar[rr]^-{\mu}           & & \Sp{3} \setminus L_f \ar[dd]^{\varphi_f} \\
                &                                                  & & \\
                & \Sp{1} \ar[rr]^{-\rho_r}                   & & \Sp{1} \\
}
\end{equation*}
where $\nu$ is the projection on the fifth coordinate, $\mu$ is the projection on the first four coordinates and $p \colon (L_F \setminus L') \to Q$ is defined by
\begin{equation*}
p(x_1,x_2,x_3,x_4,z) = \left( x_1,x_2,x_3,x_4, \varphi'(x_1,x_2,x_3,x_4,z) \right) \ .
\end{equation*}
Let us define $q \colon Q \to (L_F \setminus L')$. By \eqref{eq:pull-back} any $(x_1,x_2,x_3,x_4,\lambda) \in Q$ is of the form
\begin{equation*}
(x_1,x_2,x_3,x_4,e^{i\frac{\sigma+2\pi k'}{r}})
\end{equation*}
where $\sigma=\arg(-\varphi_f(x_1,x_2,x_3,x_4))$ and for some $k'$ such that $0\leq k'\leq r-1$. Then $q$ is
given by
\begin{equation*}
q(x_1,x_2,x_3,x_4,e^{i\frac{\sigma+2\pi k'}{r}}) = \left( x_1,x_2,x_3,x_4, |\!-\!\!f(x_1,x_2,x_3,x_4)|^{1/r}e^{i \frac{\theta+2\pi k'}{r}} \right)
\end{equation*}
where $\theta=\arg(-f(x_1,x_2,x_3,x_4))$. Then $q$ is the inverse map of $p$ and $L_F \setminus L'$ is diffeomorphic to $Q$.

Notice that diagram \eqref{dgcm} can also be seen as $\calm{P}$ being the pull-back of the $r$- fold cyclic covering $-\rho_r$ by $\varphi_f$, then $\calm{P}$ is itself a $r$-fold cyclic covering. 

Now, let $(x_1,x_2,x_3,x_4) \in L_f$, then $\pr{\calm{P}}(x_1,x_2,x_3,x_4)=\{(x_1,x_2,x_3,x_4,0) \in L_F \}$, \ie $\pr{\calm{P}}(x_1,x_2,x_3,x_4)$ consists of only one point. Then $\calm{P}$ from $L_F$ to $\Sp{3}$ is a $r$-fold cyclic branched covering with ramification locus $L_f$. 
\end{proof}

\section{Branched coverings of knots}

As is mentioned in \cite[\S~1]{MR2177068}, given a knot $(\Sp{3},K)$ and an integer $r \geq 2$, the $3$-manifold $M(K;r)$ which is the total space of the $r$-fold cyclic cover of $\Sp{3}$ branched along $K$ is called the \textbf{$r$-fold cyclic branched cover} of $K$.

By Theorem~\ref{brcov}, we have:

\begin{cor}\label{linkcov}
The link $L_F$ is the $r$-fold cyclic branched cover of $L_f$.
\end{cor}

%

\begin{exa}
Let $g$ be the polynomial map given by Proposition~\ref{polRudolph}. Let $G \colon (\R^4 \times \C,0) \to (\R^2,0) \cong (\C,0)$ be the polynomial given by
\begin{equation*}
G(x_1,x_2,x_3,x_4, z)=g(x_1,x_2,x_3,x_4) + z^r \ , r \in \N \ .
\end{equation*}
Then, by Corollary~\ref{linkcov}, $L_G$ is the $r$-fold cyclic branched cover of the figure-eight knot. By \cite{MR1084321,MR2177068}, the $3$-fold cyclic branched cover of the figure eight knot is the Hantzsche-Wendt Euclidean manifold and for $r > 3$, the link $L_G$ is an hyperbolic manifold.
\end{exa}

\begin{exa}
Perron's construction also gives the Borromean rings as link of a real isolated singularity of a polynomial function from $\R^4$ to $\R^2$ (see \cite{MR643562}). Let $f \colon (\R^4,0) \to (\R^2,0)$ be such polynomial map. Now, let $F \colon (\R^4 \times \C,0) \to (\R^2,0) \cong (\C,0)$ be the polynomial given by
\begin{equation}
F(x_1,x_2,x_3,x_4, z)=f(x_1,x_2,x_3,x_4) + z^r \ , r \in \N \ .
\end{equation}
Then $L_F$ is the $r$-fold cyclic branched cover of the Borromean rings. By \cite{MR1084321,MR2177068}, the $2$-fold cyclic branched cover of the Borromean rings is the Hantzsche-Wendt Euclidean manifold and for $r \geq 3$, the link $L_F$ is an hyperbolic manifold.
\end{exa}

By \cite{MR0365592}, it is known that examples of non-trivial singularities appear for functions from $\R^4$ to $\R^2$ and it would be interesting to have more concrete examples, however, it is hard enough to determine if a knot is fibred or not; other problem is to find if there is possible to realise the knot with only one real analytic function from $\R^4$ to $\R^2$ with isolated critical point and then to find the equation.

In recent years there have been found some families of real polynomial functions that can give examples in this particular case, as polar weighted homogeneous polynomials (see \cite{Cis09, Oka08}), mixed functions (\!\!\cite{oka-2009}), etc. It remains to study what kind of knots appear for functions of these kinds.


\end{document}